\documentclass[12pt]{article}%
\usepackage{amsmath}
\usepackage{amsfonts}
\usepackage{amssymb}
\usepackage{graphicx}%
\providecommand{\U}[1]{\protect\rule{.1in}{.1in}}
\newtheorem{theorem}{Theorem}
\newtheorem{acknowledgement}[theorem]{Acknowledgement}

\newtheorem{corollary}[theorem]{Corollary}

\newtheorem{proposition}[theorem]{Proposition}
\newtheorem{remark}[theorem]{Remark}

\newenvironment{proof}[1][Proof]{\noindent\textbf{#1.} }{\ \rule{0.5em}{0.5em}}
\begin{document}

\author{Diego Dominici \thanks{e-mail: dominicd@newpaltz.edu}\\Department of Mathematics\\State University of New York at New Paltz\\1 Hawk Dr.\\New Paltz, NY 12561-2443\\USA\\Phone: (845) 257-2607\\Fax: (845) 257-3571 }
\title{Asymptotic analysis of nested derivatives}
\date{}
\maketitle

\begin{abstract}
We analyze the nested derivatives of a function $\mathfrak{D}^{n}[f]\,(x)$
asymptotically, as $n\rightarrow\infty,$ using a discrete version of the ray
method. We give some examples showing the accuracy of our formulas.

\end{abstract}

Keywords: Inverse error function, asymptotic analysis, discrete ray method,
differential-difference equations, Taylor series.

MSC-class: 33B20 (Primary) 30B10, 34K25 (Secondary)

\section{Introduction}

\strut The oldest and most widely used method for computing the Taylor series
of inverse functions is the Lagrange Inversion theorem \cite{LIT}, which can
be stated as follows \cite{MR1424469}:

\begin{theorem}
Suppose that $f(z)$ is analytic at $a,$ and $f^{\prime}(a)\neq0.$ Then,
\[
f^{-1}\left(  z\right)  =a+%
{\displaystyle\sum\limits_{n=1}^{\infty}}
c_{n}\left(  z-b\right)  ^{n},
\]
on a neighborhood of $b=f(a)$, where%
\[
c_{n}=\frac{1}{n!}\frac{d^{n-1}}{dw^{n-1}}\left[  \frac{w-a}{f(w)-b}\right]
^{n}.
\]

\end{theorem}

Modifications and extensions of this formula were studied by Apostol
\cite{MR1790920}, Gessel \cite{MR894817}, Krattenthaler \cite{MR924765}, Roman
and Rota \cite{MR0485417} and Sokal \cite{MR2529395} among others.

Formulas relating the coefficients of the Taylor series of a function and its
inverse were obtained by Jacobsthal \cite{MR0032014}, Rauch \cite{MR0041211}
and Ostrowski \cite{MR0083599}. The two-variable case was considered in
\cite{MR1372861}.

In \cite{MR2031140}, we derived an algorithmic approach that simplifies the
computation of the derivatives of inverse functions. We defined $\mathfrak{D}%
^{n}[f]$\thinspace$(x),$ \textit{the n}$^{th}$\textit{ nested derivative} of
the function $f(x),$ by $\mathfrak{D}^{0}[f]\,(x)=1$ and%
\begin{equation}
\mathfrak{D}^{n+1}[f]\,(x)=\frac{d}{dx}\left[  f(x)\mathfrak{D}^{n}%
[f]\,(x)\right]  ,\quad n=0,1,\ldots. \label{nested}%
\end{equation}
Using these, we proved the result:

\begin{theorem}
Let $h(x)$ be analytic at $x_{0},$ and
\[
f(x)=\frac{1}{h^{\prime}(x)},\quad\left\vert f(x_{0})\right\vert \in\left(
0,\infty\right)  .
\]
$\ $ \ Then,%
\[
h^{-1}(z)=x_{0}+f(x_{0})\sum\limits_{n=1}^{\infty}\mathfrak{D}^{n-1}%
[f]\,(x_{0})\frac{(z-z_{0})^{n}}{n!},
\]
on a neighborhood of $z_{0}=h(x_{0}).$
\end{theorem}

In this paper, we study the asymptotic behavior of the nested derivatives
$\mathfrak{D}^{n}[f]\,(x)$ for large $n,$ using a discrete version of the ray
method \cite{MR2364955}, \cite{MR2401156}, \cite{MR2583007}. As a consequence,
we obtain asymptotic approximations for the higher-order derivatives of
inverse functions.

\section{Nested derivatives}

Since (\ref{nested}) is difficult to study asymptotically, we obtained in
\cite{MR2427672} a linear relation between successive nested derivatives.

\begin{proposition}
Let
\begin{equation}
g_{n}\left(  x\right)  =\frac{\mathfrak{D}^{n}[f]\,(x)}{\left[  f\left(
x\right)  \right]  ^{n}}. \label{gn}%
\end{equation}
Then, $g_{0}\left(  x\right)  =1$ and%
\begin{equation}
g_{n+1}=g_{n}^{\prime}+\left(  n+1\right)  \omega(x)g_{n},\quad n=0,1,\ldots,
\label{ddnested}%
\end{equation}
where
\begin{equation}
\omega(x)=\frac{f^{\prime}(x)}{f(x)}. \label{omega}%
\end{equation}

\end{proposition}

As a result, we obtain the following corollary.

\begin{corollary}
Let
\[
H(x)=h^{-1}(x),\quad f(x)=\frac{1}{h^{\prime}(x)},\quad z_{0}=h(x_{0}),\text{
\ }\ \left\vert f(x_{0})\right\vert \in\left(  0,\infty\right)  .
\]
Then,$\ $
\begin{equation}
\frac{d^{n}H}{dz^{n}}(z_{0})=\left[  f(x_{0})\right]  ^{n}g_{n-1}(x_{0}),\quad
n=1,2,\ldots. \label{deriv}%
\end{equation}

\end{corollary}

Later on, we will need to know the behavior of $g_{n}\left(  x\right)  $ for a
fixed value of $n$ and large $x$.

\begin{proposition}
Suppose that
\begin{equation}
\omega(x)\sim ax^{p},\quad x\rightarrow\infty.\label{w large}%
\end{equation}
Then, for fixed $n,$ we have%
\begin{equation}
g_{n}(x)\sim\left\{
\begin{array}
[c]{c}%
\left(  -1\right)  ^{n}\frac{a}{p+1}\left(  -p-1\right)  _{n}\ x^{p-n+1},\quad
p<-1\\
\left(  a-1\right)  ^{n}\left(  \frac{a}{a-1}\right)  _{n}\ x^{-n},\quad
p=-1\\
n!a^{n}x^{pn},\quad p>-1
\end{array}
\right.  \label{xlarge}%
\end{equation}
as $x\rightarrow\infty,$ where $\left(  p\right)  _{n}$ denotes the Pochhammer
symbol defined by $\left(  p\right)  _{0}=1$ and%
\[
\left(  p\right)  _{n}=%
{\displaystyle\prod\limits_{j=0}^{n-1}}
\left(  p+j\right)  ,\quad n=1,2,\ldots.
\]

\end{proposition}

\begin{proof}
From (\ref{w large}), it follows that
\begin{equation}
g_{n}(x)\sim c_{n}x^{r_{n}},\quad x\rightarrow\infty,\label{glarge}%
\end{equation}
for some sequences $c_{n},$ $r_{n}.$ Since $g_{1}(x)=\omega(x),$ we have%
\[
c_{1}=a,\quad r_{1}=p.
\]
Using (\ref{glarge}) in (\ref{ddnested}), we get%
\begin{equation}
c_{n+1}x^{r_{n+1}}\sim c_{n}r_{n}x^{r_{n}-1}+\left(  n+1\right)
c_{n}ax^{r_{n}+p},\quad x\rightarrow\infty.\label{glarge1}%
\end{equation}
Comparing powers of $x$ in (\ref{glarge1}), we obtain%
\[
r_{n+1}=\left\{
\begin{array}
[c]{c}%
r_{n}-1,\quad p<-1\\
r_{n}+p,\quad p>-1
\end{array}
\right.  .
\]

Let's consider these cases one at the time.

\begin{enumerate}
\item $p<-1$

Solving
\[
r_{n+1}=r_{n}-1,\quad r_{1}=p,
\]
we obtain
\begin{equation}
r_{n}=p+1-n. \label{rn1}%
\end{equation}
Using (\ref{rn1}) in (\ref{glarge1}), we must have%
\[
c_{n+1}=\left(  p+1-n\right)  c_{n},\quad c_{1}=a,
\]
and therefore%
\[
c_{n}=a\left(  p+1\right)  ^{-1}\left(  -1\right)  ^{n}\left(  -p-1\right)
_{n}.
\]

\item $p>-1$

Solving
\[
r_{n+1}=r_{n}+p,\quad r_{1}=p,
\]
we obtain%
\begin{equation}
r_{n}=pn. \label{rn2}%
\end{equation}
Using (\ref{rn2}) in (\ref{glarge1}), we need%
\[
c_{n+1}=\left(  n+1\right)  ac_{n},\quad c_{1}=a,
\]
and hence%
\[
c_{n}=n!a^{n}.
\]

\item $p=-1$

We have $r_{n}=-n$ and
\[
c_{n+1}=\left[  -n+\left(  n+1\right)  a\right]  c_{n},\quad c_{1}=a,
\]
which gives
\[
c_{n}=\left(  a-1\right)  ^{n}\left(  \frac{a}{a-1}\right)  _{n}.
\]

\end{enumerate}
\end{proof}

\begin{remark}
Since%
\begin{align*}
\left(  a-1\right)  ^{n}\left(  \frac{a}{a-1}\right)  _{n} &  =%
{\displaystyle\prod\limits_{j=0}^{n-1}}
\left[  \left(  a-1\right)  \left(  \frac{a}{a-1}+j\right)  \right]  \\
&  =%
{\displaystyle\prod\limits_{j=0}^{n-1}}
\left[  a+\left(  a-1\right)  j\right]  ,
\end{align*}
we have%
\[
\left(  a-1\right)  ^{n}\left(  \frac{a}{a-1}\right)  _{n}\rightarrow
1,\quad\text{as }a\rightarrow1.
\]
Hence, when $p=-1$ and $a=1,$ we see that%
\begin{equation}
g_{n}(x)\sim x^{-n}.\label{a=1}%
\end{equation}

\end{remark}

In \cite{Knessl} we analyzed the family of polynomials generated by
(\ref{ddnested}) with $\omega(x)=x.$

\section{Asymptotic analysis of $g_{n}(x)$}

We seek an approximate solution for (\ref{ddnested}) of the form
\begin{equation}
G_{n}(x)\sim\kappa\exp\left[  F(x,n)+G(x,n)\right]  ,\quad n\rightarrow
\infty\label{anszat}%
\end{equation}
where $\kappa$ is a constant and
\[
G=o(F),\quad n\rightarrow\infty.
\]
Since $G_{0}(x)=1,$ we require that%
\begin{equation}
F(x,0)=0\text{ \ \ and \ \ }G(x,0)=0.\label{initial}%
\end{equation}
Using (\ref{anszat}) in (\ref{ddnested}), we have%
\begin{gather}
\exp\left(  F+\frac{\partial F}{\partial n}+\frac{1}{2}\frac{\partial^{2}%
F}{\partial n^{2}}+G+\frac{\partial G}{\partial n}\right)  \label{asymp1}\\
=\left(  \frac{\partial F}{\partial x}+\frac{\partial G}{\partial x}\right)
\exp\left(  F+G\right)  +\left(  n+1\right)  \omega(x)\exp\left(  F+G\right)
,\nonumber
\end{gather}
where we have used%
\[
F(x,n+1)=F(x,n)+\frac{\partial F}{\partial n}(x,n)+\frac{1}{2}\frac
{\partial^{2}F}{\partial n^{2}}(x,n)+\cdots.
\]

From (\ref{asymp1}) we obtain, to leading order, the \textit{eikonal} equation%
\begin{equation}
\frac{\partial F}{\partial x}+\left(  n+1\right)  \omega(x)-\exp\left(
\frac{\partial F}{\partial n}\right)  =0, \label{eikonal}%
\end{equation}
and
\[
\exp\left(  \frac{1}{2}\frac{\partial^{2}F}{\partial n^{2}}+\frac{\partial
G}{\partial n}\right)  -\frac{\partial G}{\partial x}\exp\left(
-\frac{\partial F}{\partial n}\right)  -1=0,
\]
or, to leading order, the \textit{transport }equation%
\begin{equation}
\frac{1}{2}\frac{\partial^{2}F}{\partial n^{2}}+\frac{\partial G}{\partial
n}-\frac{\partial G}{\partial x}\exp\left(  -\frac{\partial F}{\partial
n}\right)  =0. \label{transport}%
\end{equation}

\subsection{The rays}

To solve (\ref{eikonal}), we use the method of characteristics, which we
briefly review. Given the first order partial differential equation%
\[
\mathfrak{F}\left(  x,n,F,p,q\right)  =0,\text{ \ \ with \ \ }\ p=\frac
{\partial F}{\partial x},\quad q=\frac{\partial F}{\partial n},
\]
we search for a solution \ $F(x,n)$ by solving the system of \textquotedblleft
characteristic equations\textquotedblright\
\begin{align*}
\frac{dx}{dt} &  =\frac{\partial\mathfrak{F}}{\partial p},\quad\frac{dn}%
{dt}=\frac{\partial\mathfrak{F}}{\partial q},\\
\frac{dp}{dt} &  =-\frac{\partial\mathfrak{F}}{\partial x}-p\frac
{\partial\mathfrak{F}}{\partial F},\quad\frac{dq}{dt}=-\frac{\partial
\mathfrak{F}}{\partial n}-q\frac{\partial\mathfrak{F}}{\partial F},\\
\frac{dF}{dt} &  =p\frac{\partial\mathfrak{F}}{\partial p}+q\frac
{\partial\mathfrak{F}}{\partial q},
\end{align*}
with initial conditions%
\begin{equation}
\mathfrak{F}\left[  x(0,s),n(0,s),F(0,s),p(0,s),q(0,s)\right]
=0,\label{initial1}%
\end{equation}
and%
\begin{equation}
\quad\frac{d}{ds}F(0,s)=p(0,s)\frac{d}{ds}x(0,s)+q(0,s)\frac{d}{ds}%
n(0,s),\label{initial2}%
\end{equation}
where we now consider $\left\{  x,n,F,p,q\right\}  $ to all be functions of
the variables $\left(  t,s\right)  .$

For the eikonal equation (\ref{eikonal}), we have%
\begin{equation}
\mathfrak{F}\left(  x,n,F,p,q\right)  =p-e^{q}+\left(  n+1\right)  \omega(x)
\label{eikonal1}%
\end{equation}
and therefore the characteristic equations are%
\begin{equation}
\frac{dx}{dt}=1,\quad\frac{dn}{dt}=-e^{q},\quad\frac{dp}{dt}=-\left(
n+1\right)  \omega^{\prime}(x),\quad\frac{dq}{dt}=-\omega(x), \label{charac1}%
\end{equation}
and%
\begin{equation}
\frac{dF}{dt}=p-qe^{q}. \label{eqf}%
\end{equation}
Solving (\ref{charac1}) subject to the initial conditions%
\[
x(0,s)=s,\quad n(0,s)=0,\quad q\left(  0,s\right)  =A(s),\quad p(0,s)=B(s)
\]
with $A(s),B(s)$ to be determined, we obtain%
\begin{gather*}
x(t,s)=t+s,\quad n(t,s)=\exp\left[  A(s)\right]  f(s)\left[
h(s)-h(t+s)\right]  ,\\
p(t,s)=\exp\left[  A(s)\right]  \left[  \frac{f(s)}{f(t+s)}-1\right]
+\omega(s)-(n+1)\omega(t+s)+B(s),\\
q(t,s)=\ln\left[  \frac{f(s)}{f(t+s)}\right]  +A(s).
\end{gather*}
From (\ref{initial1}) we have $B-e^{A}+\omega(s)=0$ and therefore
\[
B=e^{A}-\omega(s).
\]
Thus,%
\begin{gather}
x(t,s)=t+s,\quad n(t,s)=\exp\left[  A(s)\right]  f(s)\left[
h(s)-h(t+s)\right]  ,\nonumber\\
p(t,s)=\exp\left[  A(s)\right]  \frac{f(s)}{f(t+s)}-(n+1)\omega
(t+s),\label{char1}\\
q(t,s)=\ln\left[  \frac{f(s)}{f(t+s)}\right]  +A(s).\nonumber
\end{gather}

Since (\ref{initial}) implies that
\begin{equation}
F(0,s)=0, \label{f(0,s)}%
\end{equation}
we have from (\ref{initial2}) and (\ref{char1})%
\[
\left[  e^{A}-\omega(s)\right]  \times1+A\times0=0.
\]
Hence, $A(s)=\ln\left[  \omega(s)\right]  $ and therefore%
\begin{equation}
x=t+s,\quad n=f^{\prime}(s)\left[  h(s)-h(t+s)\right]  , \label{rays}%
\end{equation}%
\begin{equation}
p=\frac{f^{\prime}(s)}{f(t+s)}-(n+1)\omega(t+s),\quad q=\ln\left[
\frac{f^{\prime}(s)}{f(t+s)}\right]  . \label{pq}%
\end{equation}

\subsection{The functions $F$ and $G$}

Using (\ref{pq}) in (\ref{eqf}) we have%
\begin{equation}
\frac{dF}{dt}=\frac{f^{\prime}(s)}{f(t+s)}-(n+1)\omega(t+s)-\ln\left[
\frac{f^{\prime}(s)}{f(t+s)}\right]  \frac{f^{\prime}(s)}{f(t+s)}.
\label{eqf1}%
\end{equation}
Solving (\ref{eqf1}) subject to (\ref{f(0,s)}), we obtain%
\begin{equation}
F(t,s)=\ln\left[  \frac{f(s)}{f(t+s)}\right]  -n-n\ln\left[  \frac
{f(t+s)}{f^{\prime}(s)}\right]  \label{f(t,s)}%
\end{equation}
or, using (\ref{rays}),%
\begin{equation}
F=\ln\left[  \frac{f(s)}{f(x)}\right]  -n-n\ln\left[  \frac{f(x)}{f^{\prime
}(s)}\right]  . \label{f(x,n,s)}%
\end{equation}

To solve the transport equation (\ref{transport}), we need to compute
$\frac{\partial^{2}F}{\partial n^{2}},\frac{\partial G}{\partial n}$ and
$\frac{\partial G}{\partial x}$ as functions of $t,s.$ Use of the chain rule
gives%
\[%
\begin{bmatrix}
\frac{\partial x}{\partial t} & \frac{\partial x}{\partial s}\\
\frac{\partial n}{\partial t} & \frac{\partial n}{\partial s}%
\end{bmatrix}%
\begin{bmatrix}
\frac{\partial t}{\partial x} & \frac{\partial t}{\partial n}\\
\frac{\partial s}{\partial x} & \frac{\partial s}{\partial n}%
\end{bmatrix}
=%
\begin{bmatrix}
1 & 0\\
0 & 1
\end{bmatrix}
\]
and hence,%
\begin{equation}%
\begin{bmatrix}
\frac{\partial t}{\partial x} & \frac{\partial t}{\partial n}\\
\frac{\partial s}{\partial x} & \frac{\partial s}{\partial n}%
\end{bmatrix}
=\frac{1}{J(t,s)}%
\begin{bmatrix}
\frac{\partial n}{\partial s} & -\frac{\partial x}{\partial s}\\
-\frac{\partial n}{\partial t} & \frac{\partial x}{\partial t}%
\end{bmatrix}
, \label{inversion}%
\end{equation}
where the Jacobian $J(t,s)$ is defined by
\begin{equation}
J\left(  t,s\right)  =\frac{\partial x}{\partial t}\frac{\partial n}{\partial
s}-\frac{\partial x}{\partial s}\frac{\partial n}{\partial t}=\frac{\partial
n}{\partial s}-\frac{\partial n}{\partial t}. \label{J}%
\end{equation}
Using (\ref{rays}), we find that%
\begin{equation}
J(t,s)=n\frac{f^{\prime\prime}(s)}{f^{\prime}(s)}+\omega(s). \label{J1}%
\end{equation}

Using $q=\frac{\partial F}{\partial n}$ in (\ref{transport}), we have%
\[
\frac{1}{2}\frac{\partial q}{\partial n}+\frac{\partial G}{\partial n}%
-\frac{\partial G}{\partial x}e^{-q}=0
\]
or%
\[
\frac{\partial}{\partial n}\left(  \frac{1}{2}e^{q}\right)  =\frac{\partial
G}{\partial x}-\frac{\partial G}{\partial n}e^{q}%
\]
and using (\ref{charac1}), we obtain%
\[
\frac{\partial}{\partial n}\left(  \frac{1}{2}e^{q}\right)  =\frac{\partial
G}{\partial x}\frac{\partial x}{\partial t}+\frac{\partial G}{\partial n}%
\frac{\partial n}{\partial t}=\frac{\partial G}{\partial t}.
\]
Since $-e^{q}=\frac{\partial n}{\partial t},$ we have%
\begin{gather*}
\frac{\partial}{\partial n}\left(  \frac{1}{2}e^{q}\right)  =-\frac{1}{2}%
\frac{\partial}{\partial n}\left(  \frac{\partial n}{\partial t}\right)
=-\frac{1}{2}\left(  \frac{\partial^{2}n}{\partial t^{2}}\frac{\partial
t}{\partial n}+\frac{\partial^{2}n}{\partial t\partial s}\frac{\partial
s}{\partial n}\right)  \\
=-\frac{1}{2J}\left(  -\frac{\partial^{2}n}{\partial t^{2}}\frac{\partial
x}{\partial s}+\frac{\partial^{2}n}{\partial t\partial s}\frac{\partial
x}{\partial t}\right)  =-\frac{1}{2J}\left(  -\frac{\partial^{2}n}{\partial
t^{2}}+\frac{\partial^{2}n}{\partial t\partial s}\right)  \\
=-\frac{1}{2J}\frac{\partial}{\partial t}\left(  \frac{\partial n}{\partial
s}-\frac{\partial n}{\partial t}\right)  =-\frac{1}{2J}\frac{\partial
J}{\partial t},
\end{gather*}
where we have used (\ref{inversion}) and (\ref{J}). Thus,%
\[
\frac{\partial G}{\partial t}=-\frac{1}{2J}\frac{\partial J}{\partial t}%
\]
and therefore%
\[
G(t,s)=-\frac{1}{2}\ln(J)+C(s)
\]
for some function $C(s).$ Since from (\ref{initial}) we have $G(0,s)=0,$ while
(\ref{J1}) gives $J(0,s)=\omega(s),$ we conclude that $C(s)=\frac{1}{2}%
\ln\left[  \omega(s)\right]  $ and hence
\begin{equation}
G(t,s)=\frac{1}{2}\ln\left(  \frac{\left[  f^{\prime}(s)\right]  ^{2}}{\left[
f^{\prime}(s)\right]  ^{2}+nf(s)f^{\prime\prime}(s)}\right)  .\label{g(n,s)}%
\end{equation}
Replacing (\ref{f(x,n,s)}) and (\ref{g(n,s)}) in (\ref{anszat}), we obtain
$g_{n}(x)\sim\kappa\Phi\left(  x,n;s\right)  $ as $n\rightarrow\infty,$ with
\[
\Phi\left(  x,n;s\right)  =\frac{f(s)}{f(x)}e^{-n}\left[  \frac{f^{\prime}%
(s)}{f(x)}\right]  ^{n}\sqrt{\frac{\left[  f^{\prime}(s)\right]  ^{2}}{\left[
f^{\prime}(s)\right]  ^{2}+nf(s)f^{\prime\prime}(s)}}%
\]
and $\kappa$ is still to be determined. Eliminating $t$ from (\ref{rays}) we
get%
\[
n-f^{\prime}(s)\left[  h(s)-h(x)\right]  =0,
\]
which defines the function $s(x,n)$ implicitly. In cases where there exist
multiple solutions $s_{1},s_{2},\ldots,$ we must add all the contributions.

We summarize the results of this section in the following theorem.

\begin{theorem}
\label{theorem}Let the functions $g_{n}(x)$ be defined by%
\[
g_{n+1}=g_{n}^{\prime}+\left(  n+1\right)  \omega(x)g_{n},
\]
with $g_{0}(x)=1$ and%
\[
\omega(x)=\frac{d}{dx}\ln\left[  f(x)\right]  .
\]
Then, we have%
\begin{equation}
g_{n}(x)\sim\kappa\sum_{j}\Phi\left[  x,n;s_{j}(x,n)\right]  ,\quad
n\rightarrow\infty\label{Pasympt}%
\end{equation}
where
\begin{equation}
\Phi\left(  x,n;s\right)  =\frac{f(s)}{f(x)}e^{-n}\left[  \frac{f^{\prime}%
(s)}{f(x)}\right]  ^{n}\sqrt{\frac{\left[  f^{\prime}(s)\right]  ^{2}}{\left[
f^{\prime}(s)\right]  ^{2}+nf(s)f^{\prime\prime}(s)}}, \label{Phi}%
\end{equation}
$\kappa$ is an overall constant to be determined by matching and $s_{j}(x,n)$
is a solution of the equation
\begin{equation}
n-f^{\prime}(s)\left[  h(s)-h(x)\right]  =0. \label{EqS}%
\end{equation}

\end{theorem}

\section{Examples}

\begin{enumerate}
\item The natural logarithm.

Let $h(x)=\ln(x+1).$ Then,
\begin{equation}
f(x)=\frac{1}{h^{\prime}(x)}=x+1 \label{fex1}%
\end{equation}
and
\[
\omega(x)=\left(  x+1\right)  ^{-1}.
\]
In this case, (\ref{ddnested}) takes the form
\begin{equation}
g_{n+1}=g_{n}^{\prime}+\frac{n+1}{x+1}g_{n},\quad g_{0}=1. \label{gnex1}%
\end{equation}

Using (\ref{fex1}) in (\ref{Phi}), we have%
\[
\Phi\left(  x,n;s\right)  =\left(  s+1\right)  e^{-n}\left(  x+1\right)
^{-(n+1)},
\]
while (\ref{EqS}) gives%
\[
n-\ln\left(  \frac{s+1}{x+1}\right)  =0
\]
or%
\[
s=\left(  x+1\right)  e^{n}-1.
\]
Thus, from (\ref{Pasympt}), we obtain%
\[
g_{n}(x)\sim\kappa\left(  x+1\right)  ^{-n},\quad n\rightarrow\infty.
\]

Since
\[
\omega(x)=\left(  x+1\right)  ^{-1}\sim x^{-1},\quad x\rightarrow\infty,
\]
we know from (\ref{a=1}) that $g_{n}(x)\sim x^{-n}$ and therefore $\kappa=1.$
We conclude that%
\[
g_{n}(x)\sim\left(  x+1\right)  ^{-n},\quad n\rightarrow\infty.
\]
But in fact, $g_{n}(x)=\left(  x+1\right)  ^{-n}$ is the exact solution of
(\ref{gnex1})!

\item The arctangent.

Let $h(x)=\arctan(x).$ Then,
\begin{equation}
f(x)=x^{2}+1 \label{ftan}%
\end{equation}
and
\[
\omega(x)=\frac{2x}{x^{2}+1}.
\]
In this case, (\ref{ddnested}) takes the form
\[
g_{n+1}=g_{n}^{\prime}+\left(  n+1\right)  \frac{2x}{x^{2}+1}g_{n},\quad
g_{0}=1.
\]

Using (\ref{ftan}) in (\ref{Phi}), we have%
\begin{equation}
\Phi\left(  x,n;s\right)  =\frac{s^{2}+1}{x^{2}+1}e^{-n}\left[  \frac
{2s}{x^{2}+1}\right]  ^{n}\sqrt{\frac{2s^{2}}{2s^{2}+n\left(  s^{2}+1\right)
}}, \label{Phi1}%
\end{equation}
while (\ref{EqS}) gives%
\begin{equation}
n-2s\left[  \arctan(s)-\arctan(x)\right]  =0. \label{Eqs1}%
\end{equation}
For every point $\left(  x,n\right)  $ there exist two solutions $s_{-}<0$ and
$s_{+}>0$ of (\ref{Eqs1}) (see Figure 1). Hence, we get from (\ref{Pasympt})
\[
g_{n}(x)\sim\kappa\left[  \Phi\left(  x,n;s_{-}\right)  +\Phi\left(
x,n;s_{+}\right)  \right]  ,\quad n\rightarrow\infty.
\]

\begin{figure}[ptb]
\begin{center}
\includegraphics{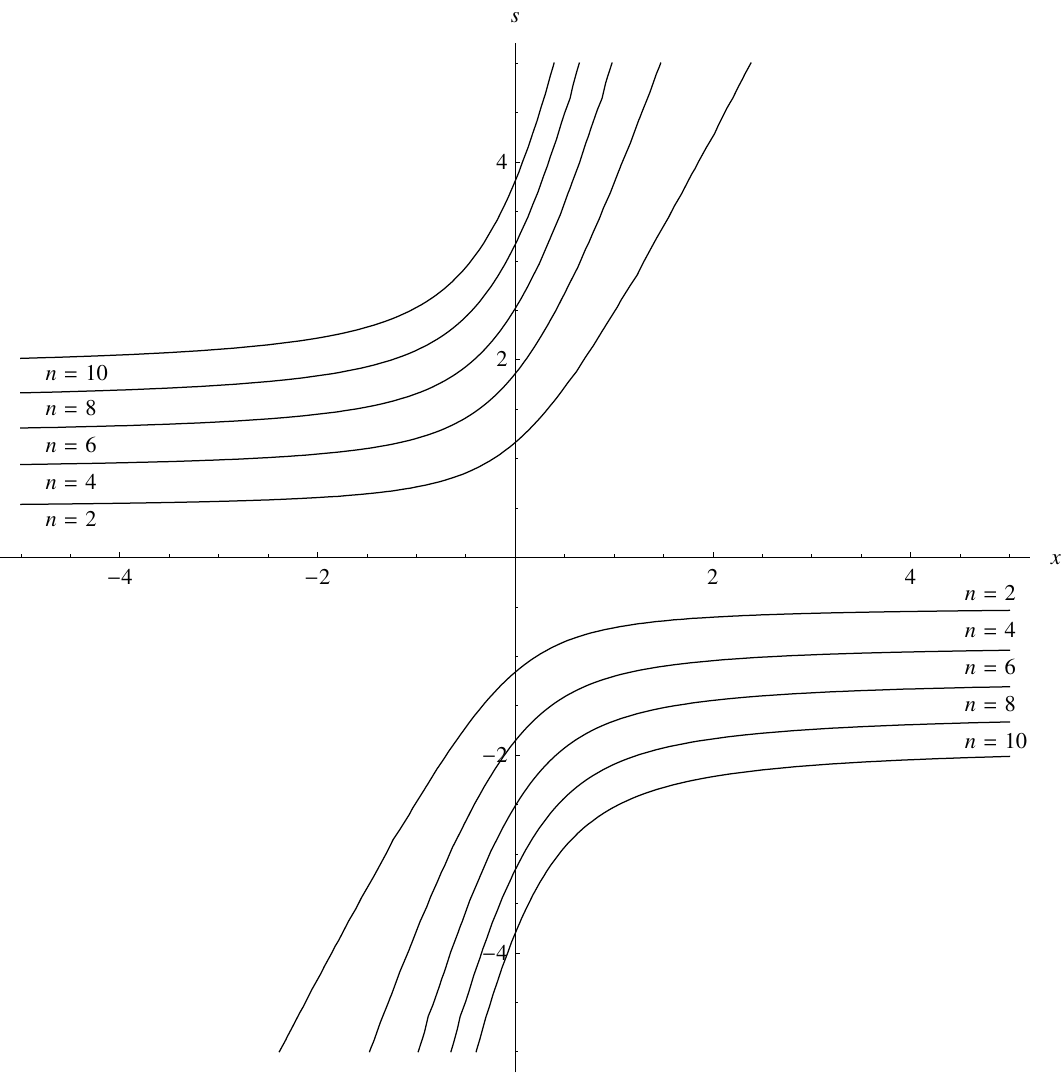}
\end{center}
\caption{A sketch of $s_{-}(x, n)$ and $s_{+}(x, n)$ for various values of
$n$. }%
\label{Fig1}%
\end{figure}

If we fix $n=5$ and let $x\rightarrow\infty,$ we get from (\ref{Eqs1})%
\[%
\begin{tabular}
[c]{|l|l|}\hline
$x$ & $s$\\\hline
$10$ & $-1.0877,\ 35.114$\\\hline
$20$ & $-1.0686,\ 70.057$\\\hline
$50$ & $-1.0575,\ 175.02$\\\hline
$100$ & $-1.0538,\ 350.01$\\\hline
\end{tabular}
\ \
\]
It follows that one solution approaches a fixed negative value $s_{-}$ and the
other $s_{+}$ grows algebraically. After some calculations, we find that%
\begin{equation}
s_{+}=\left(  1+\frac{n}{2}\right)  x+O\left(  x^{-1}\right)  ,\quad
x\rightarrow\infty. \label{sxlarge}%
\end{equation}
Using (\ref{sxlarge}) in (\ref{Phi1}) we obtain, to leading order,%
\begin{equation}
\Phi\left(  x,n;s_{+}\right)  \sim\kappa2^{-\frac{3}{2}}\left(  n+2\right)
^{n+\frac{3}{2}}e^{-n}x^{-n},\quad x\rightarrow\infty. \label{Phi1xlarge}%
\end{equation}

But since%
\[
\omega(x)=\frac{2x}{x^{2}+1}\sim2x^{-1},\quad x\rightarrow\infty
\]
we have from (\ref{xlarge})%
\begin{equation}
g_{n}\left(  x\right)  \sim\left(  n+1\right)  !x^{-n},\quad x\rightarrow
\infty. \label{g1xlarge}%
\end{equation}
Matching (\ref{Phi1xlarge}) with (\ref{g1xlarge}), we get%
\[
\kappa=2^{\frac{3}{2}}\left(  n+2\right)  ^{-n-\frac{3}{2}}e^{n}\left(
n+1\right)  !
\]
Note that%
\[
\kappa=4\sqrt{\pi}e^{-2}+O\left(  n^{-1}\right)  ,\quad n\rightarrow\infty.
\]
We conclude that%
\[
g_{n}(x)\sim4\sqrt{\pi}e^{-2}\left[  \Phi\left(  x,n;s_{-}\right)
+\Phi\left(  x,n;s_{+}\right)  \right]  ,\quad n\rightarrow\infty.
\]

If $x=0,$ (\ref{Eqs1}) becomes%
\begin{equation}
n-2s\arctan(s)=0, \label{x=0}%
\end{equation}
and solving for $s,$ we get%
\[%
\begin{tabular}
[c]{|l|l|}\hline
$n$ & $s$\\\hline
$5$ & $\pm2.1887$\\\hline
$10$ & $\pm3.8056$\\\hline
$20$ & $\pm6.9985$\\\hline
$50$ & $\pm16.551$\\\hline
$100$ & $\pm32.467$\\\hline
\end{tabular}
\ .\ \
\]
It follows that in this case $s_{-}=-s_{+},$ and therefore%
\begin{align*}
g_{n}(0)  &  \sim4\sqrt{\pi}e^{-2}\left[  \Phi\left(  0,n;s_{+}\right)
+\Phi\left(  0,n;-s_{+}\right)  \right] \\
&  =4\sqrt{\pi}\left[  1+\left(  -1\right)  ^{n}\right]  \left(  s_{+}%
^{2}+1\right)  e^{-\left(  n+2\right)  }\left(  2s_{+}\right)  ^{n}\sqrt
{\frac{2s_{+}^{2}}{2s_{+}^{2}+n\left(  s_{+}^{2}+1\right)  }}%
\end{align*}
as $n\rightarrow\infty.$ Since $g_{n}(0)=0$ for odd $n,$ we focus our
attention on even values of $n.$ From (\ref{x=0}) we obtain%
\[
s_{+}\sim\frac{n+2}{\pi},\quad n\rightarrow\infty,
\]
and hence,%
\[
g_{2n}(0)\sim\frac{4^{2\left(  n+1\right)  }}{\pi^{2n+\frac{3}{2}}%
}e^{-2\left(  n+1\right)  }\left(  n+1\right)  ^{2n+1}\frac{4\left(
n+1\right)  ^{2}+\pi^{2}}{\sqrt{4\left(  n+1\right)  ^{3}+\pi^{2}n}},\quad
n\rightarrow\infty.
\]

But since in this case $H(x)=\tan(x),$ we know that \cite{MR2031140}
\[
\mathfrak{D}^{2n}[x^{2}+1]\,(0)=\frac{2}{n+1}4^{n}\left(  4^{n+1}-1\right)
\left\vert B_{2\left(  n+1\right)  }\right\vert ,
\]
where $B_{n}$ are the Bernoulli numbers. It follows that%
\[
\left\vert B_{2\left(  n+1\right)  }\right\vert \sim2\frac{4^{n+1}}%
{\pi^{2n+\frac{3}{2}}\left(  4^{n+1}-1\right)  }\left(  \frac{n+1}{e}\right)
^{2\left(  n+1\right)  }\frac{4\left(  n+1\right)  ^{2}+\pi^{2}}%
{\sqrt{4\left(  n+1\right)  ^{3}+\pi^{2}n}},
\]
as $n\rightarrow\infty,$ or%
\[
\left\vert B_{2n}\right\vert \sim2\frac{4^{n}}{\pi^{2n-\frac{1}{2}}\left(
4^{n}-1\right)  }\left(  \frac{n}{e}\right)  ^{2n}\frac{4n^{2}+\pi^{2}}%
{\sqrt{4n^{3}+\pi^{2}\left(  n-1\right)  }},\quad n\rightarrow\infty.
\]
To leading order in $n,$ we obtain the well known asymptotic approximation%
\[
\left\vert B_{2n}\right\vert \sim4\sqrt{n\pi}\left(  \frac{n}{e\pi}\right)
^{2n},\quad n\rightarrow\infty.
\]

\end{enumerate}

\begin{acknowledgement}
This work was partially supported by a Provost Research Award from SUNY New Paltz.
\end{acknowledgement}

\newif\ifabfull\abfullfalse
\input apreambl

\end{document}